\documentclass[12pt]{amsart}
\usepackage{amsmath,mathtools,amsthm,amsfonts,bigints,amssymb, mathrsfs, tikz-cd, xcolor}
\usepackage{tabularx}
\usepackage{multirow}
\newtheorem{theorem}{Theorem}[section]
\newtheorem{lemma}[theorem]{Lemma}

\theoremstyle{definition}

\newtheorem{remark}[theorem]{Remark}

\newcommand{\R}{\mathbb R}%
\newcommand{\C}{\mathbb C}%
\newcommand{\Z}{\mathbb Z}%
\newcommand{\N}{\mathbb N}%
\newcommand{\Sc}{\mathcal S}%

\newcommand{\J}{\mathscr J}%
\newcommand{\Jc}{\mathcal J}%
\newcommand{\Hb}{\mathbb H}%
\newcommand{\X}{\mathbb X}%
\numberwithin{equation}{section}
\makeatletter

\renewcommand\subsubsection{\@secnumfont}{\bfseries}%
\renewcommand\subsubsection{\@startsection{subsubsection}{3}
  \z@{.5\linespacing\@plus.7\linespacing}{-.5em}%
  {\normalfont\bfseries}}

  \makeatother
\makeatletter
\@namedef{subjclassname@2020}{%
  \textup{2020} Mathematics Subject Classification}
\makeatother

\begin{document}

\title[Regularity and pointwise convergence]{Regularity and pointwise convergence for dispersive equations on $\mathbb{H}^2$}

\author[Utsav Dewan]{Utsav Dewan}
\address{Stat-Math Unit, Indian Statistical Institute, 203 B. T. Rd., Kolkata 700108, India}
\email{utsav\_r@isical.ac.in\:,\: utsav97dewan@gmail.com}

\subjclass[2020]{Primary 43A85, 22E30; Secondary 35J10, 43A90}

\keywords{Pointwise convergence, Dispersive equations on Riemannian Symmetric Spaces, Helgason-Fourier transform, Spherical function.}

\begin{abstract}
In the prototypical setting of non-Euclidean geometry, the 2-dimensional Real Hyperbolic space $\mathbb{H}^2$, we consider the Carleson's problem for the Schr\"odinger equation and improve the best known result until now by proving that the Sobolev regularity threshold $\beta \ge 1/2$ for the initial data, is sufficient to obtain pointwise convergence of the solution a.e. on $\mathbb{H}^2$. In fact, we prove the same bound for a wide class of dispersive equations that include the fractional Schr\"odinger equations with convex phase, the Boussinesq equation and the Beam equation, also known as the fourth order Wave equation. For the Schr\"odinger equation, we improve the result of Wang-Zhang (Canad J Math 71(4), 983-995, 2019) and for the fractional Schr\"odinger equations with convex phase, we improve the result of Cowling (Lecture Notes Math 992, 83-90, 1983). 
\end{abstract}

\maketitle

\section{Introduction}
One of the most celebrated problems in Euclidean Harmonic analysis is the Carleson's problem: determining the optimal regularity of the initial condition $f$ of the
Schr\"odinger equation given by
\begin{equation} \label{schrodinger}
\begin{cases}
	 i\frac{\partial u}{\partial t} -\Delta_{\R^n} u=0\:,\:\:\:  (x,t) \in \mathbb{R}^n \times \mathbb{R}\:, \\
	u(0,\cdot)=f\:, \text{ on } \mathbb{R}^n \:,
	\end{cases}
\end{equation}
in terms of the index $\beta$ such that for all $f$ belonging to the inhomogeneous Sobolev space $H^\beta(\mathbb{R}^n)$, the solution $u$ of (\ref{schrodinger}) converges pointwise to $f$, 
\begin{equation} \label{pointwise_convergence}
\displaystyle\lim_{t \to 0+} u(x,t)=f(x)\:,\:\:\text{ almost everywhere }.
\end{equation}
In 1980, Carleson studied this problem for $n=1$ in \cite{C} and obtained the sufficient condition $\beta \ge 1/4$ for the pointwise convergence (\ref{pointwise_convergence}) to hold. In 1982, Dahlberg-Kenig \cite{DK} showed that $\beta \ge 1/4$ is also necessary. This completely solved the problem in dimension one and subsequently, posed the question in higher dimensions.

\medskip

In 1983, Cowling \cite{Cowling} studied the problem for a general class of self-adjoint operators on $L^2(X)$ for a measure space $X$  and obtained $\beta > 1$, to be a sufficient condition for the associated Schr\"odinger operator. For $\R^n$, in 1987, Sj\"olin \cite{Sjolin} improved the bound to $\beta > 1/2$, by means of a local smoothing effect. This improvement was also independently obtained by Vega \cite{Vega} in 1988. Then special attention was focused towards the case $n=2$ and developing decoupling techniques, continuous improvements were made \cite{Bourgain1, MVV, Tao, Lee}. These attempts reached fruition recently in 2017, when Du-Guth-Li \cite{DGL} obtained the bound $\beta > 1/3$, for $\R^2$ by using polynomial partitioning and decoupling. Then in 2019, Du-Zhang \cite{DZ} obtained the bound $\beta > n/2(n+1)$, for dimensions $n \ge 3$, by using decoupling and induction on scales. This bound is sharp except at the endpoint $\beta=n/2(n+1)$, due to a counterexample by Bourgain \cite{Bourgain}. Thus in the Euclidean setting, the Carleson's problem has been almost fully resolved.

\medskip

In the case of a Riemannian manifold, where the volume measure neither satisfies any doubling property nor does it admit dilations, the above mentioned Euclidean machineries work no longer and thus it offers a fresh challenge. The prototypical example of such a manifold is the $2$-dimensional Real Hyperbolic space $\Hb^2$. On $\Hb^2$, the Schr\"odinger equation corresponding to the Laplace-Beltrami operator $\Delta$, is given by,
\begin{equation} \label{schrodinger_H2}
\begin{cases}
	 i\frac{\partial u}{\partial t} -\Delta u=0\:,\:\:\:  (x,t) \in \Hb^2 \times \mathbb{R}\:, \\
	u(0,\cdot)=f\:, \text{ on } \Hb^2 \:.
	\end{cases}
\end{equation}
In fact, the equation (\ref{schrodinger_H2}) is just a special case of the family of fractional Schr\"odinger equations with convex phase of degree $a>1$:
\begin{equation} \label{frac_schrodinger}
\begin{cases}
	 i\frac{\partial u}{\partial t} +\left(-\Delta \right)^{\frac{a}{2}} u=0\:,\:\:\:  (x,t) \in \Hb^2 \times \mathbb{R}\:,\\
	u(0,\cdot)=f\:, \text{ on } \Hb^2 \:.
	\end{cases}
\end{equation}

It is natural to study the Carleson's problem for the family of equations (\ref{frac_schrodinger}) in $\Hb^2$, for initial data belonging to the Sobolev spaces:
\begin{equation} \label{Sobolev_space}
H^\beta\left(\Hb^2\right):=\left\{f \in L^2 \left(\Hb^2\right) \mid \left(-\Delta\right)^{\frac{\beta}{2}}f \in L^2 \left(\Hb^2\right) \right\}\:,\:\beta \ge 0\:.
\end{equation}
The only known result in this direction is the sufficiency of the bound $\beta > a/2$ and is due to Cowling \cite{Cowling}. In the particular case of the Schr\"odinger equation (that is, $a=2$) however, Wang-Zhang have recently improved the bound down to $\beta > 1/2$ \cite[Theorem 1.1]{WZ}. 

\medskip

The purpose of this article is to improve the above results down to $\beta=1/2$. In fact, we will obtain the bound $\beta=1/2$, for general dispersive equations of the form,
\begin{equation} \label{dispersive}
\begin{cases}
	 i\frac{\partial u}{\partial t} +\Psi(\sqrt{-\Delta} )u=0\:,\:  (x,t) \in \Hb^2 \times \R\:, \\
	u(0,\cdot)=f\:,\: \text{ on } \Hb^2 \:,
	\end{cases}
\end{equation}
such that the phase function of the corresponding multiplier $\psi(\lambda):=\Psi\left(\sqrt{\lambda^2 + 1}\right)$ is radial, real-valued and its (radial) derivatives satisfy the following high frequency asymptotics for $|\lambda|>1$ and some $a>1$:
\begin{equation} \label{phase_fn_properties}
\begin{cases}
	 |\psi'(\lambda)| \asymp |\lambda|^{a-1}\:,\\
	 |\psi''(\lambda)| \asymp |\lambda|^{a-2}\:,\\
     |\psi'''(\lambda)| \lesssim |\lambda|^{a-3}\:. 
\end{cases}
\end{equation}
We note that for $\Psi(r):= r^a$, with $a>1$, we recover the fractional Schr\"odinger equations with convex phase (\ref{frac_schrodinger}). Some other examples are given by $\Psi(r):= r \sqrt{1+r^2}$ and $\Psi(r):= \sqrt{1+r^4}$ corresponding respectively to the Boussinesq equation and the Beam equation, also known as the fourth order Wave equation (for more details see \cite{FX, GPW}), which often appear in Mathematical Physics.

\medskip

We now present our main result:
\begin{theorem} \label{thm1}
For Dispersive equations (\ref{dispersive}), with phase satisfying (\ref{phase_fn_properties}), the pointwise convergence of the solution to its initial data,
\begin{equation*}
\displaystyle\lim_{t \to 0+} u(x,t)=f(x)\:,
\end{equation*}
holds for almost every $x \in \Hb^2$, whenever the initial data $f \in H^\beta(\Hb^2)$ with $\beta \ge 1/2$.
\end{theorem} 

To address the problem of pointwise convergence, the key is to consider the corresponding maximal function,
\begin{equation*} 
S^*_{\psi} f(x):= \displaystyle \sup_{0<t<1} \left|u(x,t)\right|\:,
\end{equation*}
where $u$ is the solution of (\ref{dispersive}) with initial data $f$. By standard arguments in the literature (for instance, see the proof of Theorem 5 of \cite{Sjolin}), Theorem \ref{thm1} then follows from the maximal estimate:
\begin{theorem}\label{thm2}
Let $f$ be an $L^2$-Schwartz class function on $\Hb^2$ and $B\left(x_0,1/2\right)$ denote the geodesic ball centred at $x_0 \in \Hb^2$ with radius $1/2$. Then for Dispersive equations (\ref{dispersive}), with phase $\psi$ satisfying (\ref{phase_fn_properties}), the maximal estimate
\begin{equation} \label{maximal_fn_inequality}
{\|S^*_{\psi}f\|}_{L^1\left(B\left(x_0,\frac{1}{2}\right)\right)} \lesssim {\|f\|}_{H^\beta(\Hb^2)}\:,
\end{equation}
holds for all $x_0 \in \Hb^2$ and all $\beta \ge 1/2$.
\end{theorem}

\begin{remark} \label{remark1}
We now briefly summarize the novelties of this article:
\begin{itemize}
\item In \cite{Cowling}, Cowling's arguments only involved abstract Hilbert space theory. In \cite{WZ}, the authors obtained the bound $\beta>1/2$, by means of a local smoothing effect, building up on Doi's ideas \cite{Doi} on the interaction between smoothing effects of the Schr\"odinger evolution group and the behaviour of the geodesic flow.

\medskip

\item The main idea in our article however is completely different: utilizing the underlying Riemannian Symmetric space structure of $\Hb^2 \cong SU(1,1)/SO(2)$, our approach is based on the notion of Helgason-Fourier transform, arising from the representation theory of semi-simple Lie groups. Then crucially using a local Bessel series expansion of the spherical function $\varphi_\lambda$ (given by Lemma \ref{bessel_series_expansion}), we estimate the joint oscillation afforded by $\varphi_\lambda$ and the Fourier multiplier, to improve the bound down to $\beta=1/2$. 
\end{itemize}
\end{remark}

This article is organised as follows. In Section $2$, we recall the necessary preliminaries and fix our notations. In Section $3$, we prove Theorem \ref{thm2}. Finally, in Section $4$, we conclude by sketching some possible directions for further research.
 
\section{Preliminaries}
In this section, we recall some preliminaries and fix our notations. Throughout, the symbols `c' and `C' will denote positive constants whose values may change on each occurrence. The enumerated constants $C_1,C_2, \dots$ will however be fixed throughout. $\N$ will denote the set of positive integers. For $x \in \R$, $\lceil x\rceil$ will denote the smallest integer $m \in \Z$ such that $x \le m$. For non-negative functions $f_1,\:f_2$ we write,  $f_1 \lesssim f_2$ if there exists a constant $C \ge 1 $, so that
\begin{equation*}
f_1 \le C f_2 \:\:,
\end{equation*}
and $f_1 \asymp f_2$ if 
\begin{equation*}
\frac{1}{C} f_1 \le f_2 \le C f_1\:.
\end{equation*}
For a real-valued function $f_3$, we will also write $f_3=\mathcal{O}(f_1)$ if 
\begin{equation*}
|f_3| \lesssim f_1\:.
\end{equation*}

We now recall the required preliminaries about rank one Riemannian Symmetric spaces of non-compact type and Fourier analysis thereon. The relevant information can be found in \cite{GV, HelGGA, HelSymm}.

Let $G$ be a connected, non-compact, semi-simple Lie group with finite center and $K$ be a maximal compact subgroup of $G$. We consider the Riemannian Symmetric space of non-compact type $\X=G/K$. Let $n$ be the dimension of $\X$. Thus elements $x \in \X$ are of the form $x=gK$, for $g \in G$. We let $o=eK$ denote the origin of $\X$ and $dx$ the left $G$-invariant Riemannian volume measure on $\X$. The group $G$ acts on $\X$ by isomteries, that is, the geodesic distance function $d(\cdot,\cdot)$ on $\X$, is left $G$-invariant:
\begin{equation*}
d\left(g \cdot x, g \cdot y\right)=d(x,y)\:, x,y \in \X\:,\: g \in G\:.
\end{equation*}

Expressing the Iwasawa decomposition in the form $G=NAK$, any element $g \in G$ can be uniquely written as, $g=n(g)\exp\left(A(g)\right)k(g)$, where $A(g) \in \mathfrak{a}$, the Lie algebra of $A$. Let $M$ and $M'$ be the centralizer and normalizer of $A$ in $K$ respectively. Let $\mathbb{B}=K/M$ and $db$ denote the natural probability measure on it. The finite group $W=M/M'$ is called the Weyl group. The function $A$ defined above actually defines a function on $\X \times \mathbb{B}$ via,
\begin{equation*}
A\left(gK,kM\right)=A(k^{-1}g)\:.
\end{equation*}

The dimension of $\mathfrak{a}$ is called the rank of $G$ and also of $\X$. Hence, when rank = $1$, $\mathfrak{a}^*$ and $\mathfrak{a}^*_\C$, the real  dual of $\mathfrak{a}$ and its complexification respectively, can be identified with $\R$ and $\C$. From now onwards, we focus only in the rank one case. We fix some order on the non-zero restricted roots. There are at most two roots which are positive relative to this order. These roots will be denoted by $\alpha$ and $2\alpha$. Letting $p$ and $q$ denote the multiplicities of these roots respectively, we set $\rho=(p+2q)/2$.

One also has the Cartan decomposition $G=K\overline{A^+}K$. Let $g^+$ be the $\overline{\mathfrak{a}^+}$-component of $g \in G$, in the decomposition $G= K\left(\exp \overline{\mathfrak{a}^+}\right) K$ and let $\sigma(g)=|g^+|$. For any $f \in L^1(\X)$, its integral with respect to the volume measure can be written in the polar coordinate as,
\begin{equation*}
\int_{\X} f(x)\:dx \:=\: \int_K \int_0^\infty f(ka_s \cdot o)\:D(s)\:ds\:dk\:,
\end{equation*}
where $D(\cdot)$ is the density function, $ds$ is the Lebesgue measure restricted to $[0,\infty)$ and $dk$ is the Haar measure on $K$.

The spherical function on $\X$ has the form,
\begin{equation*}
\varphi_\lambda(x) = \int_{\mathbb{B}} \:e^{(i\lambda+\rho)A(x,b)}\:db\:,\:\text{ for } \lambda \in \C\:,
\end{equation*}
and satisfies the functional equation:
\begin{equation} \label{functional_eqn}
\varphi_\lambda(g^{-1}_yg_x \cdot o)= \int_{\mathbb{B}} e^{(i\lambda+\rho)A(x,b)}\:e^{(-i\lambda+\rho)A(y,b)}\:db\:,
\end{equation}
where $g_x,g_y \in G$ such that $x=g_xK$ and $y=g_yK$. Moreover,
\begin{equation*}
\varphi_\lambda = \varphi_{-\lambda}\:,\: \lambda \in \C\:,
\end{equation*}
and thus in particular,
\begin{equation} \label{-i_rho}
\varphi_{-i\rho}(x)=\varphi_{i\rho}(x) \equiv 1\:,\: x \in \X\:.
\end{equation}
We also have
\begin{equation} \label{phi_lambda_bound}
\left|\varphi_\lambda(x)\right| \le 1\:,\: x \in \X\:,\: \lambda \in \R\:.
\end{equation}
The spherical function is a radial function, that is,
\begin{equation*}
\varphi_\lambda(x)=\varphi_\lambda(s)\:,
\end{equation*}
where $s=d(o,x)$, the geodesic distance of $x$ from the origin.

Although, in the generality of rank one Riemannian Symmetric spaces of non-compact type, we do not have an explicit expression of $\varphi_\lambda$, for points near the origin however,   $\varphi_\lambda$ does admit a Bessel series expansion. But before stating that result, let us define the following normalizing constant in terms of the Gamma functions,
\begin{equation*}
c_0 = \pi^{\frac{1}{2}}\: 2^{\left(\frac{q}{2}-2\right)}\: \: \frac{\Gamma\left(\frac{n-1}{2}\right)}{\Gamma\left(\frac{n}{2}\right)}\:,
\end{equation*}
and the following functions on $\C$, for all $\mu \ge 0$,
\begin{equation*}
\J_\mu(z)= 2^{\mu-1} \: \Gamma\left(\frac{1}{2} \right)\: \Gamma\left(\mu + \frac{1}{2} \right) \frac{J_\mu(z)}{z^\mu},
\end{equation*}
where $J_\mu$ are the Bessel functions \cite[p. 154]{SW}.
\begin{lemma}\cite[Theorem 2.1]{ST} \label{bessel_series_expansion}
Let $\X$ be a rank one Riemannian Symmetric space of non-compact type and for $x \in \X$, let $s=d(o,x)$. Then there exist $R_0>1, R_1>1$, such that for any $0 \le s \le R_0$, and any integer $M \ge 0$, and all $\lambda \in \R$, we have
\begin{equation*}
\varphi_\lambda(s)= c_0 {\left(\frac{s^{n-1}}{D(s)}\right)}^{1/2} \displaystyle\sum_{l=0}^M a_l(s)\J_{\frac{n-2}{2}+l}(\lambda s) s^{2l} + E_{M+1}(\lambda,s)\:,
\end{equation*}
where
\begin{equation*}
a_0 \equiv 1\:,\: |a_l(s)| \le C {(4R_1)}^{-l}\:,
\end{equation*}
and the error term has the following behaviour
\begin{equation*}
	\left|E_{M+1}(\lambda,s) \right| \le C_M \begin{cases}
	 s^{2(M+1)}  & \text{ if  }\: |\lambda s| \le 1 \\
	s^{2(M+1)} {|\lambda s|}^{-\left(\frac{n-1}{2} + M +1\right)} &\text{ if  }\: |\lambda s| > 1 \:.
	\end{cases}
\end{equation*}
\end{lemma}

We will also require the following asymptotic expansion of Bessel functions:
\begin{lemma} \label{Szego_result} \cite[Eqn (1.71.8), p. 16]{Szego}
For $\mu \ge 0$, we have for any $p \in \N$,
\begin{eqnarray*}
J_\mu(t)&=& \left(\frac{2}{\pi t}\right)^\frac{1}{2} \cos\left(t-\frac{\mu \pi}{2}-\frac{\pi}{4}\right)\left\{\sum_{l=0}^{p-1}\tilde{a}_l t^{-2l}\:+\: \mathcal{O}\left(t^{-2p}\right)\right\} \\
&& + \left(\frac{2}{\pi t}\right)^\frac{1}{2} \sin\left(t-\frac{\mu \pi}{2}-\frac{\pi}{4}\right)\left\{\sum_{l=0}^{p-1}b_l t^{-2l-1}\:+\: \mathcal{O}\left(t^{-2p-1}\right)\right\}\:,
\end{eqnarray*}
for $t \to +\infty$, where $\tilde{a}_l$ and $b_l$ are constants depending only on $l$ and also $\tilde{a}_0 \equiv 1$. 
\end{lemma} 
As a consequence, we have:
\begin{lemma} \label{Bessel_fn_expansion}
There exist complex numbers $z_1,\dots,z_6$, such that for $t \to +\infty$,
\begin{eqnarray*}
J_0(t) &=& e^{it}\left\{z_1\:t^{-\frac{1}{2}}+z_2\:t^{-\frac{3}{2}}\right\} +\: e^{-it}\left\{z_3\:t^{-\frac{1}{2}}+z_4\:t^{-\frac{3}{2}}\right\} + \mathcal{O}\left(t^{-\frac{5}{2}}\right)\:, \\
\frac{J_1(t)}{t}&=& z_5\: e^{it}\:t^{-\frac{3}{2}}\: +\: z_6\: e^{-it}\:t^{-\frac{3}{2}} + \mathcal{O}\left(t^{-\frac{5}{2}}\right)\:.
\end{eqnarray*}
\end{lemma}
\begin{proof}
The result follows at once from Lemma \ref{Szego_result}, by considering $p=2$ for $J_0$, and $p=1$ for $J_1$. 
\end{proof}
Let $\mathfrak{g}$ be the Lie algebra of $G$ and $\mathcal{U}(\mathfrak{g})$ denote its universal enveloping algebra. The elements of $\mathcal{U}(\mathfrak{g})$ act on $C^\infty(G)$ as differential operators, on both sides. Following the notation of Harish-Chandra, we shall write $f\left(D;g;E\right)$ for the action of $(D,E) \in \mathcal{U}(\mathfrak{g}) \times \mathcal{U}(\mathfrak{g})$ on $f \in C^\infty(G)$ at $g \in G$. More precisely,
\begin{eqnarray*}
f\left(D;g;E\right) &=& \left(\frac{\partial}{\partial s_1} \cdots \frac{\partial}{\partial s_d}\frac{\partial}{\partial t_1}\cdots \frac{\partial}{\partial t_e}\right)\big{|}_{s_1=\cdots=s_d=t_1=\cdots=t_e=0} \\
&& \times f\left(\left(\exp s_1X_1\right)\cdots\left(\exp s_dX_d\right)g\left(\exp t_1Y_1\right)\cdots \left(\exp t_eY_e\right)\right)\:,
\end{eqnarray*}
if $D=X_1 \cdots X_d$ and $E=Y_1 \cdots Y_e$, for $X_1, \cdots, X_d, Y_1, \cdots, Y_e \in \mathfrak{g}$.

We now come to the definition of $L^2$-Schwartz class functions. $\Sc^2(G)$, called the $L^2$-Schwartz class on $G$, is the collection of all $f \in C^\infty(G)$ such that for each $D,E \in \mathcal{U}(\mathfrak{g})$ and all non-negative integers $m$,
\begin{equation*}
\left|f\left(D;g;E\right)\right| \lesssim_m \varphi_0(g) \left(1+\sigma(g)\right)^{-m}\:,\: g \in G\:.
\end{equation*}
$\Sc^2(\X)$ is the sub-collection of functions in $\Sc^2(G)$ that are also right $K$-invariant. In particular, for all $f \in \Sc^2(\X)$ and all non-negative integers $m$, $\Delta^m f \in L^2(\X)$.
 
For $f \in \Sc^2(\X)$, the Helgason-Fourier transform of $f$ is defined as,
\begin{equation*}
\tilde{f}(\lambda,b)=\int_{\X} f(x)\:e^{(-i\lambda+\rho)A(x,b)}\:dx\:,\: (\lambda,b) \in \R \times \mathbb{B}\:. 
\end{equation*}
The function $x \mapsto e^{(-i\lambda+\rho)A(x,b)}$ is an eigenfunction of $\Delta$ with eigenvalue $-\left(\lambda^2 + \rho^2\right)$, and hence for all $f \in \Sc^2(\X)$,
\begin{equation} \label{Fourier_Delta}
\widetilde{\left(\Delta f\right)}(\lambda,b)=   -\left(\lambda^2 + \rho^2\right) \tilde{f}(\lambda,b)\:.
\end{equation}

The relevant inversion formula is given by,
\begin{equation*}
f(x) = \frac{1}{|W|}\int_{\mathbb{B}} \int_{\R}
\tilde{f}(\lambda,b)\:e^{(i\lambda+\rho)A(x,b)}\:{|{\bf c}(\lambda)|}^{-2}\: d\lambda\:db\:,\: x \in \X\:,
\end{equation*}
where ${\bf c}(\cdot)$ denotes the Harish-Chandra's ${\bf c}$-function. The Plancherel identity is given by,
\begin{equation*}
\left\|f\right\|_{L^2(\X)}\:=\: c\:\left(\int_{\mathbb{B}} \int_{\R} \left|\tilde{f}(\lambda,b)\right|^2\:{|{\bf c}(\lambda)|}^{-2}\: d\lambda\:db \right)^{\frac{1}{2}}\:,
\end{equation*}
for some $c>0$. The weight in the Plancherel measure satisfies \cite[Lemma 4.2]{ST}:
\begin{equation} \label{planceherel_measure}
{|{\bf c}(\lambda)|}^{-2} \lesssim \left(1\:+\:|\lambda|\right)^{n-1}\:,\: \lambda \in \R\:.
\end{equation}

Using the above notion of the Helgason-Fourier transform, for $f \in \Sc^2(\X)$, the solution of the equation
\begin{equation*} 
\begin{cases}
	 i\frac{\partial u}{\partial t} +\Psi(\sqrt{-\Delta} )u=0\:,\:  (x,t) \in \X \times \R\:, \\
	u(0,\cdot)=f\:,\: \text{ on } \X \:,
	\end{cases}
\end{equation*}
is given by the formula,
\begin{equation} \label{dispersive_soln}
u(x,t)= \frac{1}{|W|}\int_{\mathbb{B}} \int_{\R}
\tilde{f}(\lambda,b)\:e^{(i\lambda+\rho)A(x,b)}\:e^{it\psi(\lambda)}\:{|{\bf c}(\lambda)|}^{-2}\: d\lambda\:db\:.
\end{equation}
In view of the Plancherel identity, we can also express the norms for the Sobolev spaces defined in (\ref{Sobolev_space}) by,
\begin{equation*}
\left\|f\right\|_{H^\beta(\X)}\: \asymp\: \left(\int_{\mathbb{B}} \int_{\R} \left(\lambda^2+\rho^2\right)^\beta\:\left|\tilde{f}(\lambda,b)\right|^2\:{|{\bf c}(\lambda)|}^{-2}\: d\lambda\:db \right)^{\frac{1}{2}}\:.
\end{equation*}

In this article, we are primarily interested in $\Hb^2 \cong SU(1,1)/SO(2)$, where
\begin{eqnarray*}
 SU(1,1)\:&=&\:\left\{\begin{pmatrix}
a & b \\
\bar{b} & \bar{a}
\end{pmatrix} \mid a,b \in \C,\: |a|^2 - |b|^2=1 \right\}\:,\\
SO(2)\:&=&\:\left\{\begin{pmatrix}
e^{i\theta} & 0 \\
0 & e^{-i\theta}
\end{pmatrix} \mid \theta \in [0,2\pi)\right\}\:.
\end{eqnarray*}
In this case, the density function is given by,
\begin{equation*}
D(s)=\sinh(2s)\:,
\end{equation*}
and hence for $s<1$, satisfies the local growth asymptotics:
\begin{equation} \label{density}
D(s) \asymp s\:.
\end{equation}
Moreover, the weight of the Plancherel measure is explcitly given by, 
\begin{equation*}
{|{\bf c}(\lambda)|}^{-2}=\lambda \tanh \left(\frac{\pi \lambda}{2}\right).
\end{equation*}

\section{Proof of Theorem \ref{thm2}}
We start off the proof in the general setting of rank one Riemannian Symmetric spaces of non-compact type and then gradually specialize to $\Hb^2$.

\medskip

Let $G$ be a connected, non-compact, semi-simple Lie group with finite center and real rank one. Let $K$ be a maximal compact subgroup of $G$ and $\X =G/K$ be the corresponding Symmetric space. Let $f \in \Sc^2(\X)$, the class of $L^2$-Schwartz class functions on $\X$ and $B\left(x_0,1/2\right)$ denote the geodesic ball on $\X$, with centre $x_0 \in \X$ and radius $1/2$, with respect to the left $G$-invariant metric $d$ on $\X$.

\medskip

We note that to prove the local maximal estimate (\ref{maximal_fn_inequality}), it suffices to prove the following estimate at the end-point:
\begin{equation} \label{pf_eq1}
\left\| T_{\psi}f\right\|_{L^1\left(B\left(x_0,\frac{1}{2}\right)\right)} \lesssim {\|f\|}_{H^{\frac{1}{2}}(\X)}\:, \text{ for all } x_0 \in \X\:,
\end{equation}
where  $T_{\psi}f$ is the linearization given by the formula (\ref{dispersive_soln}),
\begin{equation*}
T_\psi f(x):=u(x,t(x))= \frac{1}{|W|}\int_{\mathbb{B}} \int_{\R}
\tilde{f}(\lambda,b)\:e^{(i\lambda+\rho)A(x,b)}\:e^{it(x)\psi(\lambda)}\:{|{\bf c}(\lambda)|}^{-2}\: d\lambda\:db\:,
\end{equation*}
for any measurable function $t(\cdot):\X \to (0,1).$

\medskip

For $f \in \Sc^2(\X)$, by (\ref{-i_rho}), (\ref{Fourier_Delta}), in view of the growth asymptotic of ${|{\bf c}(\lambda)|}^{-2}$ given in (\ref{planceherel_measure}), the Cauchy-Schwarz inequality and the Plancherel theorem yields for sufficiently large $m_0 \in \N$ and for all $x \in \X$,
\begin{eqnarray*}
\left|T_\psi f(x)\right| &\lesssim & {\|\Delta^{m_0}f\|}_{L^2(\X)}\:\left(\int_{\mathbb{B}} e^{2\rho A(x,b)}\:db \int_{\R} \frac{{|{\bf c}(\lambda)|}^{-2}}{(\lambda^2+\rho^2)^{2m_o}} d\lambda\right)^{\frac{1}{2}} \\
&=& {\|\Delta^{m_0}f\|}_{L^2(\X)}\:\varphi_{-i\rho}(x)^{\frac{1}{2}}\:\left( \int_{\R} \frac{{|{\bf c}(\lambda)|}^{-2}}{(\lambda^2+\rho^2)^{2m_o}} d\lambda\right)^{\frac{1}{2}}\\
&<& \infty\:,
\end{eqnarray*}
and hence there exists a complex-valued function $\Theta$ on $\X$ such that for all $x \in \X$,  
\begin{equation*}
\left|T_\psi f(x)\right| = T_\psi f(x) \cdot \Theta(x)\:, \text{ with } \left|\Theta (x)\right|\equiv 1\:.
\end{equation*}

\medskip

Next, by Fubini's Theorem and the Cauchy-Schwarz inequality, we note that
\begin{eqnarray*}
&&\left\| T_{\psi}f\right\|_{L^1\left(B\left(x_0,\frac{1}{2}\right)\right)} \\
&=&\int_{B\left(x_0,\frac{1}{2}\right)} T_{\psi}f(x)\cdot \Theta(x)\: dx \\
&=& \frac{1}{|W|} \int_{\mathbb{B}} \int_{\R}
\tilde{f}(\lambda,b)\left(\int_{B\left(x_0,\frac{1}{2}\right)} e^{(i\lambda+\rho)A(x,b)}\:e^{it(x)\psi(\lambda)} \Theta(x)\:dx \right){|{\bf c}(\lambda)|}^{-2}\: d\lambda\:db \\
& \lesssim & {\|f\|}_{H^{\frac{1}{2}}(\X)}\left(\int_{\mathbb{B}}\int_{\R}
\left|\int_{B\left(x_0,\frac{1}{2}\right)} e^{(i\lambda+\rho)A(x,b)}\:e^{it(x)\psi(\lambda)}\Theta(x)\:dx \right|^2 \frac{{|{\bf c}(\lambda)|}^{-2}}{|\lambda|}\: d\lambda\:db\right)^{\frac{1}{2}}\:,
\end{eqnarray*}
and hence to prove (\ref{pf_eq1}), it suffices to show that
\begin{equation} \label{pf_eq2}
\int_{\mathbb{B}}\int_{\R}
\left|\int_{B\left(x_0,\frac{1}{2}\right)} e^{(i\lambda+\rho)A(x,b)}\:e^{it(x)\psi(\lambda)}\Theta(x)\:dx \right|^2 \frac{{|{\bf c}(\lambda)|}^{-2}}{|\lambda|}\: d\lambda\:db <\infty\:.
\end{equation}
In this direction, by Fubini's theorem and the functional equation (\ref{functional_eqn}), we see that
\begin{eqnarray*}
&& \int_{\mathbb{B}}\int_{\R}
\left|\int_{B\left(x_0,\frac{1}{2}\right)} e^{(i\lambda+\rho)A(x,b)}\:e^{it(x)\psi(\lambda)}\Theta(x)\:dx \right|^2 \frac{{|{\bf c}(\lambda)|}^{-2}}{|\lambda|}\: d\lambda\:db \\
&=& \int_{\mathbb{B}}\int_{\R}
\int_{B\left(x_0,\frac{1}{2}\right)} \int_{B\left(x_0,\frac{1}{2}\right)} e^{(i\lambda+\rho)A(x,b)}\:e^{it(x)\psi(\lambda)}\:\Theta(x)\:e^{(-i\lambda+\rho)A(y,b)}\:e^{-it(y)\psi(\lambda)}\:\overline{\Theta(y)} \\
&& \:\:\:\:\:\:\:\:\:\:\:\:\:\:\:\:\:\:\:\:\:\:\:\:\:\:\:\:\:\:\:\:\:\:\:\:\:\:\:\:\:\:\:\times  \frac{{|{\bf c}(\lambda)|}^{-2}}{|\lambda|}\:dx\:dy\: d\lambda\:db \\
&=& \int_{B\left(x_0,\frac{1}{2}\right)} \int_{B\left(x_0,\frac{1}{2}\right)} \Theta(x)\:\overline{\Theta(y)} \int_{\R} e^{i(t(x)-t(y))\psi(\lambda)}\left(\int_{\mathbb{B}} e^{(i\lambda+\rho)A(x,b)}\:e^{(-i\lambda+\rho)A(y,b)}\:db\right)\\
&& \:\:\:\:\:\:\:\:\:\:\:\:\:\:\:\:\:\:\:\:\:\:\:\:\:\:\:\:\:\:\:\:\times \frac{{|{\bf c}(\lambda)|}^{-2}}{|\lambda|}d\lambda\:dx\:dy\: \\
& \le & \int_{B\left(x_0,\frac{1}{2}\right)} \int_{B\left(x_0,\frac{1}{2}\right)} \left|\int_{\R} \varphi_\lambda(g^{-1}_yg_x \cdot o)\:e^{i(t(x)-t(y))\psi(\lambda)}\frac{{|{\bf c}(\lambda)|}^{-2}}{|\lambda|}d\lambda\right|\:dx\:dy\:,
\end{eqnarray*}
where $g_x,g_y \in G$ such that $x=g_xK$ and $y=g_yK$. Now as the spherical function, the phase  and ${|{\bf c}(\cdot)|}^{-2}$ all are even, the last integral simplies to
\begin{equation*}
2\int_{B\left(x_0,\frac{1}{2}\right)} \int_{B\left(x_0,\frac{1}{2}\right)} \left|\int_0^\infty \varphi_\lambda(g^{-1}_yg_x \cdot o)\:e^{i(t(x)-t(y))\psi(\lambda)}\frac{{|{\bf c}(\lambda)|}^{-2}}{\lambda}d\lambda\right|\:dx\:dy\:.
\end{equation*}
Thus to prove (\ref{pf_eq2}), it suffices to prove that
\begin{equation} \label{pf_eq3}
\int_{B\left(x_0,\frac{1}{2}\right)} \int_{B\left(x_0,\frac{1}{2}\right)} \left|\int_0^\infty \varphi_\lambda(g^{-1}_yg_x \cdot o)\:e^{i(t(x)-t(y))\psi(\lambda)}\frac{{|{\bf c}(\lambda)|}^{-2}}{\lambda}d\lambda\right|\:dx\:dy <\infty\:.
\end{equation}
Now as $\varphi_\lambda$ is radial and the distance function on $\X$ is left $G$-invariant, we have
\begin{equation*}
\varphi_\lambda(g^{-1}_yg_x \cdot o) = \varphi_\lambda(d(o,g^{-1}_yg_x \cdot o))=\varphi_\lambda(d(g_y \cdot o,g_x \cdot o))= \varphi_\lambda(d(y,x))\:.
\end{equation*}
Also as $x,y \in B\left(x_0,1/2\right)$, We note that 
\begin{equation} \label{pf_eq4}
d(y,x) \le d(y,x_0)+d(x,x_0) < 1\:.
\end{equation}
Then for a fixed pair $(x,y) \in B\left(x_0,1/2\right) \times B\left(x_0,1/2\right)$, setting $s=d(y,x)$ and $\tau=t(x)-t(y)$, we focus on the $\lambda$-integral in (\ref{pf_eq3}), that is,
\begin{equation*}
I:= \int_0^\infty \varphi_\lambda(s)\: e^{i\tau\psi(\lambda)}\:\frac{{|{\bf c}(\lambda)|}^{-2}}{\lambda}\:d\lambda\:.
\end{equation*}
We now specialize to $\X=\Hb^2$. Thus plugging ${|{\bf c}(\lambda)|}^{-2}=\lambda \tanh \left(\frac{\pi \lambda}{2}\right)$, the above integral simplifies to
\begin{equation*}
I= \int_0^\infty \varphi_\lambda(s)\: e^{i\tau\psi(\lambda)}\:\tanh \left(\frac{\pi \lambda}{2}\right)\:d\lambda\:.
\end{equation*} 
We now estimate the above oscillatory integral by introducing a smooth resolution of identity. More precisely, let $\eta \in C^\infty_c(\R)$ be even, non-negative such that
\begin{equation*}
Supp(\eta) \subset \left\{\xi \in \R \mid 1 < |\xi| < 4\right\}\:, \text{ and } \sum_{j=-\infty}^\infty \eta\left(2^{-j}\:\xi\right)=1\:,\text{ for } \xi \ne 0\:.
\end{equation*} 
This enables us to decompose $I$ as,
\begin{equation*}
I = \sum_{j=-\infty}^\infty \int_0^\infty \eta\left(2^{-j}\:\lambda\right) \varphi_\lambda(s)\: e^{i\tau\psi(\lambda)}\:\tanh \left(\frac{\pi \lambda}{2}\right)\:d\lambda\:,
\end{equation*}
which after the change of variable $\lambda \mapsto 2^j\lambda$, takes the form
\begin{equation*}
I  =\sum_{j=-\infty}^\infty 2^j \int_1^4 \eta\left(\lambda\right) \varphi_{2^j\lambda}(s)\: e^{i\tau\psi(2^j\lambda)}\:\tanh \left(2^{j-1}\pi \lambda\right)\:d\lambda = \sum_{j=-\infty}^\infty 2^j I_j\:.
\end{equation*}
We first estimate for low frequency. By the boundedness of $\varphi_\lambda$ given by (\ref{phi_lambda_bound}) and $\tanh$, we trivially have
\begin{equation*}
|I_j| \le 3\:.
\end{equation*}
And thus 
\begin{equation} \label{pf_eq5}
\left|\sum_{j=-\infty}^{\lceil -\log_2(s)\rceil} 2^j\:I_j\right| \lesssim \sum_{j=-\infty}^{\lceil -\log_2(s)\rceil} 2^j \lesssim s^{-1}\:.
\end{equation}
We now focus on high frequency, that is, on $j > \lceil -\log_2(s)\rceil$ which is equivalent to saying $2^j > s^{-1}$. By (\ref{pf_eq4}), $s <1$. Moreover since, $\lambda \in (1,4)$, consequently we have
\begin{equation*}
2^j\lambda s > 2^js>1\:.
\end{equation*}
Then by Lemma \ref{bessel_series_expansion}, for $M=1$, we have
\begin{equation} \label{pf_eq6}
\varphi_{2^j\lambda}(s)=c_0\left(\frac{s}{D(s)}\right)^{\frac{1}{2}} \left[J_0\left(2^j\lambda s\right)\:+\:a_1(s)s^2\frac{J_1\left(2^j\lambda s\right)}{2^j\lambda s}\right] + \mathcal{E}_1\left(2^j\lambda, s\right)\:,
\end{equation}
where 
\begin{equation} \label{pf_eq7}
\left|\mathcal{E}_1\left(2^j\lambda, s\right)\right| \lesssim s^4 \left(2^j\lambda s\right)^{-\frac{5}{2}}\:.
\end{equation}
Now by Lemma \ref{Bessel_fn_expansion}, we have complex numbers $z_1,\dots,z_6$ such that
\begin{eqnarray}\label{pf_eq8}
J_0\left(2^j\lambda s\right) &=& e^{i2^j\lambda s}\left\{z_1\left(2^j\lambda s\right)^{-\frac{1}{2}}+z_2\left(2^j\lambda s\right)^{-\frac{3}{2}}\right\} \nonumber\\
&& +\: e^{-i2^j\lambda s}\left\{z_3\left(2^j\lambda s\right)^{-\frac{1}{2}}+z_4\left(2^j\lambda s\right)^{-\frac{3}{2}}\right\} + \mathcal{E}_2(2^j\lambda s)\:, 
\end{eqnarray}
where 
\begin{equation} \label{pf_eq9}
\left|\mathcal{E}_2(2^j\lambda s)\right| \lesssim (2^j\lambda s)^{-\frac{5}{2}}\:.
\end{equation}
and
\begin{equation} \label{pf_eq10}
\frac{J_1\left(2^j\lambda s\right)}{2^j\lambda s}= z_5\: e^{i2^j\lambda s}\left(2^j\lambda s\right)^{-\frac{3}{2}}\: +\: z_6\: e^{-i2^j\lambda s}\left(2^j\lambda s\right)^{-\frac{3}{2}} + \mathcal{E}_3(2^j\lambda s)\:, 
\end{equation}
where 
\begin{equation} \label{pf_eq11}
\left|\mathcal{E}_3(2^j\lambda s)\right| \lesssim (2^j\lambda s)^{-\frac{5}{2}}\:.
\end{equation}
Therefore, plugging (\ref{pf_eq8})-(\ref{pf_eq11}) in (\ref{pf_eq6}), we get that
\begin{eqnarray} \label{pf_eq12}
\varphi_{2^j\lambda}(s)&=& e^{i2^j\lambda s}\left\{\theta_1(s)\left(2^j\lambda s\right)^{-\frac{1}{2}}\:+\:\theta_2(s)\left(2^j\lambda s\right)^{-\frac{3}{2}}\right\} \nonumber\\
&& +\: e^{-i2^j\lambda s}\left\{\theta_3(s)\left(2^j\lambda s\right)^{-\frac{1}{2}}\:+\:\theta_4(s)\left(2^j\lambda s\right)^{-\frac{3}{2}}\right\} + \mathcal{E}(2^j\lambda, s)\:,
\end{eqnarray}
where the complex-valued functions $\theta_j$ (for $j=1, \cdots, 4$) satisfy,
\begin{equation*}
|\theta_j(s)| \lesssim 1\:,\:\text{ for } s<1\:,
\end{equation*}
and 
\begin{equation*}
\mathcal{E}(2^j\lambda, s) = \mathcal{E}_1(2^j\lambda, s) +  c_0\left(\frac{s}{D(s)}\right)^{\frac{1}{2}} \left[\mathcal{E}_2(2^j\lambda s)\:+\: a_1(s)\:s^2\:\mathcal{E}_3(2^j\lambda s)\right]\:.
\end{equation*}
Thus by (\ref{pf_eq7}), (\ref{pf_eq9}), (\ref{pf_eq11}) and the local growth asymptotics of the density function given by (\ref{density}), it follows that
\begin{equation} \label{pf_eq13}
\left|\mathcal{E}(2^j\lambda, s)\right| \lesssim (2^j\lambda s)^{-\frac{5}{2}}\:.
\end{equation}
Thus for $j > \lceil -\log_2(s)\rceil$, invoking the decomposition (\ref{pf_eq12}) in the definition of $I_j$, we write,
\begin{eqnarray*}
I_j &=& \theta_1(s)\:(2^js)^{-\frac{1}{2}}\int_1^4 e^{i\{2^j\lambda s \:+\: \tau\psi(2^j\lambda)\}}\: \eta\left(\lambda\right)\:\lambda^{-\frac{1}{2}}\: \tanh \left(2^{j-1}\pi \lambda\right)\:d\lambda \\
&&+\: \theta_2(s)\:(2^js)^{-\frac{3}{2}}\int_1^4 e^{i\{2^j\lambda s \:+\: \tau\psi(2^j\lambda)\}}\: \eta\left(\lambda\right)\:\lambda^{-\frac{3}{2}}\: \tanh \left(2^{j-1}\pi \lambda\right)\:d\lambda \\
&&+ \:\theta_3(s)\:(2^js)^{-\frac{1}{2}}\int_1^4 e^{i\{-2^j\lambda s \:+\: \tau\psi(2^j\lambda)\}}\: \eta\left(\lambda\right)\:\lambda^{-\frac{1}{2}}\: \tanh \left(2^{j-1}\pi \lambda\right)\:d\lambda \\
&&+\: \theta_4(s)\:(2^js)^{-\frac{3}{2}}\int_1^4 e^{i\{-2^j\lambda s \:+\: \tau\psi(2^j\lambda)\}}\: \eta\left(\lambda\right)\:\lambda^{-\frac{3}{2}}\: \tanh \left(2^{j-1}\pi \lambda\right)\:d\lambda \\
&&+\: \int_1^4 \eta\left(\lambda\right) \mathcal{E}(2^j\lambda, s)\: e^{i\tau\psi(2^j\lambda)}\:\tanh \left(2^{j-1}\pi \lambda\right)\:d\lambda\: \\
&=& \sum_{m=1}^5 I^m_j\:.
\end{eqnarray*}
For $I^5_j$, by the pointwise estimate (\ref{pf_eq13}), we have
\begin{equation*}
\left|I^5_j\right| \lesssim (2^js)^{-\frac{5}{2}}\:,
\end{equation*}
and hence as $2^j>s^{-1}$, we get that
\begin{equation} \label{pf_eq14}
\left|\sum_{j>\lceil -\log_2(s)\rceil}^\infty 2^j\:I^5_j\right| \lesssim s^{-\frac{5}{2}}\sum_{j>\lceil -\log_2(s)\rceil}^\infty 2^{-\frac{3j}{2}} \lesssim s^{-\frac{5}{2}}\: s^{\frac{3}{2}}  = s^{-1}\:.
\end{equation}
We next take care of $I^m_j$ for $m=1,\dots,4$, simultaneously, by considering for $l=0$ or $1$, the oscillatory integral
\begin{equation*}
\int_1^4 e^{iF_s(\lambda)}\:H(\lambda)\:d\lambda\:,
\end{equation*}
where
\begin{equation*}
F_s(\lambda)=\pm 2^j\lambda s + \tau \psi\left(2^j \lambda\right)\:, \text{ and } H(\lambda)= \eta(\lambda)\:\lambda^{-\left(\frac{1}{2}+l\right)}\:\tanh \left(2^{j-1}\pi \lambda\right)\:.
\end{equation*}
We note that
\begin{eqnarray} \label{pf_eq15}
&& F'_s(\lambda)=\pm 2^js\:+\: 2^j \tau \psi'\left(2^j \lambda\right)\:,\\
\label{pf_eq16}
&& F''_s(\lambda)= 2^{2j} \tau \psi''\left(2^j \lambda\right)\:, \\
\label{pf_eq17}
&& F'''_s(\lambda)= 2^{3j} \tau \psi'''\left(2^j \lambda\right)\:.
\end{eqnarray}
We now recall the properties of the phase function $\psi$ given by (\ref{phase_fn_properties}): for $\lambda \in (1,4)$ and $j > \lceil -\log_2(s)\rceil$, there exist $a>1$ and positive constants $C_1,C_2,C_3 \ge 1$, such that
\begin{eqnarray} \label{pf_eq18}
&&\frac{1}{C_1} \left(2^j \lambda\right)^{a-1} \le \left|\psi'\left(2^j \lambda\right)\right| \le C_1 \left(2^j \lambda\right)^{a-1}\:, \\
\label{pf_eq19}
&&\frac{1}{C_2} \left(2^j \lambda\right)^{a-2} \le \left|\psi''\left(2^j \lambda\right)\right| \le C_2 \left(2^j \lambda\right)^{a-2}\:, \\
\label{pf_eq20}
&&\left|\psi'''\left(2^j \lambda\right)\right| \le C_3 \left(2^j \lambda\right)^{a-3}\:.
\end{eqnarray}
Thus combining (\ref{pf_eq16}) and (\ref{pf_eq17}) with the upper bounds in (\ref{pf_eq19}) and (\ref{pf_eq20}), it follows that
\begin{eqnarray} \label{pf_eq21}
&&\left|F''_s(\lambda)\right| \lesssim 2^{aj}|\tau|\:, \\
\label{pf_eq22}
&&\left|F'''_s(\lambda)\right| \lesssim 2^{aj}|\tau|\:.
\end{eqnarray}

Setting $\zeta(\lambda)=\eta(\lambda)\:\lambda^{-\left(\frac{1}{2}+l\right)}$, we also have
\begin{eqnarray*}
 H'(\lambda)&=&\zeta'(\lambda)\tanh \left(2^{j-1}\pi \lambda\right)\:+\:\left(2^{j-1}\pi\right)\zeta(\lambda)\cosh^{-2} \left(2^{j-1}\pi \lambda\right)\:,\\
H''(\lambda)&=&\zeta''(\lambda)\tanh \left(2^{j-1}\pi \lambda\right)\:+\:\left(2^j\pi\right)\zeta'(\lambda)\cosh^{-2} \left(2^{j-1}\pi \lambda\right)\\
&&+\: \left(2^{2j-1}\pi^2\right)\zeta(\lambda)\cosh^{-2} \left(2^{j-1}\pi \lambda\right)\tanh \left(2^{j-1}\pi \lambda\right)                  ,
\end{eqnarray*}
and hence for $\lambda \in (1,4)$ and $j > \lceil -\log_2(s)\rceil$,
\begin{equation} \label{pf_eq23}
\left|H^{(p)}(\lambda)\right| \lesssim 1\:,\:\text{ for } p=0,1,2\:,
\end{equation}
with the implicit constants being independent of $j$. 

We now proceed with the oscillatory integral estimate by first considering
\begin{equation*}
\Jc_1 := \left\{j > \lceil -\log_2(s)\rceil \mid 2^js \le 2^{(aj-1)}C^{-1}_1 |\tau|\right\}\:.
\end{equation*}
For $j \in \Jc_1$, by (\ref{pf_eq15}), the triangle inequality, the lower bound in (\ref{pf_eq18}) and the facts that $\lambda \in (1,4)$ and $a>1$, it follows that
\begin{eqnarray} \label{pf_eq24}
\left|F'_s(\lambda)\right| & \ge & 2^j|\tau|\left|\psi'\left(2^j\lambda\right)\right| - 2^js \nonumber\\
& \ge & 2^j C^{-1}_1|\tau|\left(2^j \lambda\right)^{a-1} - 2^js \nonumber\\
& > & 2^{aj} C^{-1}_1|\tau| - 2^js \nonumber\\
& \ge & 2^{aj} C^{-1}_1|\tau| - 2^{(aj-1)}C^{-1}_1 |\tau| \nonumber\\
&=& 2^{(aj-1)}C^{-1}_1 |\tau|\:.
\end{eqnarray} 
Then by performing integration-by-parts twice and using (\ref{pf_eq21})-(\ref{pf_eq24}), we obtain
\begin{equation*}
\left|\int_1^4 e^{iF_s(\lambda)}\:H(\lambda)\:d\lambda\right| \lesssim \left(2^{aj}|\tau|\right)^{-2}\:,
\end{equation*}
which yields for $m=1,\dots,4$ and $l=0,1$,
\begin{eqnarray} \label{pf_eq25}
\left|\sum_{j \in \Jc_1} 2^j\:I^m_j\right| & \lesssim & s^{-\left(\frac{1}{2}+l\right)} \sum_{j \in \Jc_1}  2^j\:  2^{-j\left(\frac{1}{2}+l\right)}\:\left(2^{aj}|\tau|\right)^{-2} \nonumber\\
& \lesssim & s^{-\left(\frac{1}{2}+l\right)} \sum_{j \in \Jc_1}  2^j\:  2^{-j\left(\frac{1}{2}+l\right)}\:\left( 2^js\right)^{-2} \nonumber\\
& \le & s^{-\left(\frac{5}{2}+l\right)} \sum_{j>\lceil -\log_2(s)\rceil}^\infty 2^{-j\left(\frac{3}{2}+l\right)} \nonumber\\
& \lesssim & s^{-\left(\frac{5}{2}+l\right)}\: s^{\left(\frac{3}{2}+l\right)} \nonumber\\
&=& s^{-1}\:.
\end{eqnarray}
Next, we focus on 
\begin{equation*}
\Jc_2 := \left\{j > \lceil -\log_2(s)\rceil \mid 2^js \ge 2^{(aj+1)}4^{a-1}C_1 |\tau|\right\}\:.
\end{equation*}
For $j \in \Jc_2$, by (\ref{pf_eq15}), the triangle inequality, the upper bound in (\ref{pf_eq18}) and the facts that $\lambda \in (1,4)$ and $a>1$, it follows that
\begin{eqnarray} \label{pf_eq26}
\left|F'_s(\lambda)\right| & \ge & 2^js - 2^j|\tau|\left|\psi'\left(2^j\lambda\right)\right| \nonumber\\
& \ge & 2^js - 2^j C_1 |\tau|\left(2^j \lambda\right)^{a-1} \nonumber\\
& > &  2^js - 2^{aj}4^{a-1}C_1|\tau| \nonumber\\
& \ge & 2^js - 2^{j-1}s \nonumber\\
&=& 2^{j-1}s\:.
\end{eqnarray}
Then by performing integration-by-parts twice and using (\ref{pf_eq21})-(\ref{pf_eq23}) and (\ref{pf_eq26}), while keeping in mind the definition of $\Jc_2$, we obtain
\begin{equation*}
\left|\int_1^4 e^{iF_s(\lambda)}\:H(\lambda)\:d\lambda\right| \lesssim \left(2^js\right)^{-2}\:,
\end{equation*}
which yields for $m=1,\dots,4$ and $l=0,1$,
\begin{eqnarray} \label{pf_eq27}
\left|\sum_{j \in \Jc_2} 2^j\:I^m_j\right|
& \lesssim & s^{-\left(\frac{1}{2}+l\right)} \sum_{j \in \Jc_2}  2^j\:  2^{-j\left(\frac{1}{2}+l\right)}\:\left( 2^js\right)^{-2} \nonumber\\
& \le & s^{-\left(\frac{5}{2}+l\right)} \sum_{j>\lceil -\log_2(s)\rceil}^\infty 2^{-j\left(\frac{3}{2}+l\right)} \nonumber\\
& \lesssim & s^{-\left(\frac{5}{2}+l\right)}\: s^{\left(\frac{3}{2}+l\right)} \nonumber\\
&=& s^{-1}\:.
\end{eqnarray}
Finally, we consider
\begin{equation*}
\Jc_3 := \left\{j > \lceil -\log_2(s)\rceil \mid 2^{(aj-1)}C^{-1}_1 |\tau| < 2^js < 2^{(aj+1)}4^{a-1}C_1 |\tau|\right\}\:.
\end{equation*}
From the above definition, it immediately follows that the set $\Jc_3$ is finite, moreover its cardinality is bounded by, 
\begin{equation} \label{pf_eq28}
\left|\Jc_3\right| \le \left(\frac{2}{a-1}\right)\log_2(2C_1)\:+\:2\:.
\end{equation}
Now for $j > \lceil -\log_2(s)\rceil$ and $\lambda \in (1,4)$, by (\ref{pf_eq16}) and the lower bound in (\ref{pf_eq19}), we have
\begin{equation*}
\left|F''_s(\lambda)\right| \ge \left(C^{-1}_2 \min\left\{1\:,\:4^{a-2}\right\}\right)2^{aj}|\tau|\:.
\end{equation*}
Combining this with (\ref{pf_eq23}), it follows from Van der Corput (for instance, see the proof of \cite[Proposition 2, pp. 332-333]{BigStein}) that
\begin{equation*}
\left|\int_1^4 e^{iF_s(\lambda)}\:H(\lambda)\:d\lambda\right| \lesssim \left(2^{aj}|\tau|\right)^{-\frac{1}{2}}\:.
\end{equation*}
The above inequality along with the facts that $2^js>1$ and (\ref{pf_eq28}), while keeping in mind the definition of $\Jc_3$, then yields for $m=1,\dots,4$ and $l=0,1$,
\begin{eqnarray} \label{pf_eq29}
\left|\sum_{j \in \Jc_3} 2^j\:I^m_j\right|
& \lesssim &s^{-\left(\frac{1}{2}+l\right)} \sum_{j \in \Jc_3}  2^j\:  2^{-j\left(\frac{1}{2}+l\right)}\:\left(2^{aj}|\tau|\right)^{-\frac{1}{2}} \nonumber\\
& \lesssim &s^{-\left(\frac{1}{2}+l\right)} \sum_{j \in \Jc_3}  2^j\:  2^{-j\left(\frac{1}{2}+l\right)}\:\left(2^js\right)^{-\frac{1}{2}} \nonumber\\
&=& s^{-1} \sum_{j \in \Jc_3} \left(2^js\right)^{-l} \nonumber\\
& \le & s^{-1} \: \left|\Jc_3\right| \nonumber\\
& \lesssim & s^{-1} \:.
\end{eqnarray}

So combining (\ref{pf_eq5}), (\ref{pf_eq14}), (\ref{pf_eq25}),  (\ref{pf_eq27}) and (\ref{pf_eq29}), we obtain
\begin{equation*}
\left|I\right| \lesssim  s^{-1} \:,
\end{equation*}
which in turn yields
\begin{eqnarray*} 
&&\int_{B\left(x_0,\frac{1}{2}\right)} \int_{B\left(x_0,\frac{1}{2}\right)} \left|\int_0^\infty \varphi_\lambda(g^{-1}_yg_x \cdot o)\:e^{i(t(x)-t(y))\psi(\lambda)}\frac{{|{\bf c}(\lambda)|}^{-2}}{\lambda}d\lambda\right|\:dx\:dy \\
& \lesssim & \int_{B\left(x_0,\frac{1}{2}\right)} \int_{B\left(x_0,\frac{1}{2}\right)} \frac{dx\:dy}{d(x,y)}\:.
\end{eqnarray*}
By the left $G$-invariance of the distance function and the volume measure, we have
\begin{equation*}
\int_{B\left(x_0,\frac{1}{2}\right)} \int_{B\left(x_0,\frac{1}{2}\right)} \frac{dx\:dy}{d(x,y)} = \int_{B\left(o,\frac{1}{2}\right)} \int_{B\left(o,\frac{1}{2}\right)} \frac{dx\:dy}{d(x,y)}\:.
\end{equation*}
Then by Fubini-Tonelli's theorem, another application of the left $G$-invariance of the distance function and the volume measure and the local growth asymptotics of the density function (\ref{density}), it follows that
\begin{eqnarray*}
\int_{B\left(o,\frac{1}{2}\right)} \int_{B\left(o,\frac{1}{2}\right)} \frac{dx\:dy}{d(x,y)} &=& \int_{B\left(o,\frac{1}{2}\right)} \left(\int_{B\left(o,\frac{1}{2}\right)} \frac{dy}{d(x,y)}\right)dx \\
& \le & \int_{B\left(o,\frac{1}{2}\right)} \left(\int_{B\left(x,1\right)} \frac{dy}{d(x,y)}\right)dx \\
& = & \int_{B\left(o,\frac{1}{2}\right)} \left(\int_{B\left(o,1\right)} \frac{dy}{d(o,y)}\right)dx \\
& \lesssim & 1\:,
\end{eqnarray*} 
and hence, we obtain (\ref{pf_eq3}). This completes the proof of Theorem \ref{thm2}. \qed

\section{Concluding remarks}
We now sketch some possible directions for further research:
\begin{enumerate}
\item On $\Hb^2$, Theorem \ref{thm1} establishes the sufficiency of the bound $\beta \ge 1/2$. On the other hand, for radial initial data, the Carleson's problem for dispersive equations on $\Hb^2$ has recently been fully resolved \cite{Dewan, DR, Dewan2} and in particular, yields the necessity of the regularity threshold $\beta \ge 1/4$. So it will be interesting to seek the critical Sobolev index $\beta_0 \in [1/4,1/2]$ for solving the Carleson's problem. For the particular case of the Schr\"odinger equation at least, keeping in mind the results for $\R^2$ \cite{Bourgain, DGL}, a reasonable guess would be to conjecture $\beta_0=1/3$. 

\medskip

\item It would be interesting to see whether the sufficiency of the bound $\beta \ge 1/2$ continues to hold true for higher dimensional Real Hyperbolic spaces or more generally, for the class of rank one Riemannian Symmetric spaces of non-compact type.
\end{enumerate}

\section*{Acknowledgements} 
The author would like to thank Swagato K. Ray for suggestions. The author is supported by a research fellowship of Indian Statistical Institute.

\bibliographystyle{amsplain}

\end{document}